\documentclass[10pt,oneside]{article}
 \usepackage[a4paper,centering,bindingoffset=0cm,hmargin=3cm,vmargin=3cm,includeheadfoot]{geometry}
\usepackage[font=footnotesize,labelfont=bf,width=0.75\textwidth,labelsep=period]{caption}
\usepackage{amsmath}
\usepackage{amssymb,bbm,color}
\usepackage{graphicx}
\usepackage{authblk}
\usepackage{bm}
  \allowdisplaybreaks[1]
  \graphicspath{{images/}}
  \usepackage{enumerate}

\usepackage[pdftex,colorlinks=true,linkcolor=blue,unicode,pdfpagelabels]{hyperref}

\hypersetup
{
    pdfauthor={Drossos Gintides, Leonidas Mindrinos},
    pdftitle={The inverse scattering problem by an elastic inclusion},
}

\usepackage[hyperref,amsmath,thmmarks]{ntheorem}
\usepackage{aliascnt}

\let\RE\Re
\let\Re=\undefined
\DeclareMathOperator{\Re}{\RE e}
\let\IM\Im
\let\Im=\undefined
\DeclareMathOperator{\Im}{\IM m}
\newcommand{\R}{\mathbbm R}

\renewcommand{\C}{\mathbbm C}

\renewcommand{\i}{\mathrm i}
\renewcommand{\b}{\textbf }
\newcommand{\n}{\bm{\hat{n}}}
\newcommand{\ta}{\bm{\hat{\tau}}}
\newcommand{\di}{\bm{\hat{d}}}
\newcommand{\xhat}{\widehat{\bm{x}}}
\newcommand{\ka}{\bm\kappa}
\newcommand{\mi}{\bm\mu}
\newcommand{\zita}{\bm\zeta}
\newcommand{\ksi}{\bm\xi}
\newcommand{\bb}[1]{\bm{\mathcal{#1}}}

\newcommand{\norm}[1]{\lVert #1 \rVert}

\newaliascnt{proposition}{lemma}

\aliascntresetthe{proposition}

\newaliascnt{theorem}{lemma}
\newtheorem{theorem}[theorem]{Theorem}
\aliascntresetthe{theorem}

\newaliascnt{assumption}{lemma}
\newtheorem{assumption}[assumption]{Iterative Scheme}
\aliascntresetthe{assumption}

\newaliascnt{remark}{lemma}
\newtheorem{remark}[remark]{Remark}
\aliascntresetthe{remark}

\theoremstyle{nonumberplain}
\theoremseparator{:}
\theoremheaderfont{\normalfont\itshape}
\theorembodyfont{\normalfont}
\theoremsymbol{\ensuremath{\square}}
\newtheorem{proof}{Proof}

\makeatletter
\newcommand{\logmessage}[1]{\@latex@warning{#1}}
\makeatother

\providecommand{\keywords}[1]{\small\textbf{Keywords} #1}

\begin{document}

\title{The inverse scattering problem by an elastic inclusion}

\author[1]{Roman Chapko\thanks{chapko@lnu.edu.ua}}
\author[2]{Drossos Gintides\thanks{dgindi@math.ntua.gr}}
\author[3]{Leonidas Mindrinos\thanks{leonidas.mindrinos@univie.ac.at}}
\affil[1]{\small Faculty of Applied Mathematics and Informatics, Ivan Franko National University of Lviv, Ukraine.}
\affil[2]{Department of Mathematics, National Technical University of Athens, Greece.}
\affil[3]{Computational Science Center, University of Vienna,  Austria.}

\renewcommand\Authands{ and }
\normalsize

\date{}

\maketitle

\begin{abstract}
In this work we consider the inverse elastic scattering problem by an inclusion in two dimensions. The elastic inclusion is placed in an isotropic homogeneous elastic medium. The inverse problem, using the third Betti's formula (direct method), is equivalent to a system of four integral equations that are non linear with respect to the unknown boundary. Two equations are on the boundary and two on the unit circle where the far-field patterns of the scattered waves lie. We solve iteratively the system of integral equations by linearising only the far-field equations. Numerical results are presented that illustrate the feasibility of the proposed method.

\vspace{0.2cm}
\keywords{linear elasticity, inverse scattering problem, integral equation method}
\end{abstract}

\section{Introduction}
The inverse scattering problem consists on finding the shape and the location of an obstacle by measuring the scattered wave, close or far from the scatterer. Depending on the kind of illumination (acoustic, electromagnetic or elastic) and the properties of the obstacle (soft, hard, penetrable or not) one faces different kind of problems regarding the unique solvability of the problem and the numerical scheme for approximating the solution. 

In this work we place the obstacle in a two-dimensional homogeneous and isotropic elastic medium and we assume that it is penetrable, a so-called inclusion, with different Lam\'e parameters from the exterior domain. We consider as incident wave an elastic longitudinal or transversal wave that after interacting with the boundary of the medium is split into an interior (transmitted) and a scattered wave, propagating in the inclusion and the exterior, respectively.  The scattered wave is also decomposed into a longitudinal and a transversal wave with different wavenumbers that behave like spherical waves with different polarizations at infinity. 

Before considering the inverse problem, we should have a good knowledge of the direct problem, which is to find the scattered field and its far-field patterns from the knowledge of the obstacle and the incident wave. The direct problem is linear and
well posed for smooth obstacles \cite{Mar90}. The inverse problem that we consider here can be seen as a continuation of \cite{GinMin11} where the inverse problem was examined for a rigid scatterer and a cavity. The problem of detecting an elastic inclusion has been also considered for given boundary measurements \cite{AleMorRos04, AlvMar09}, using the factorization method \cite{ChaKirAnaGinKir07}, the linear sampling method \cite{PelSev03, Sev05}, a gradient descent method \cite{Lou15, PelKleBer00} or the probing method \cite{KarSin15}.

Here, we solve this inverse problem by formulating an equivalent system of non-linear integral equations that has to be solved with a regularization iterative scheme due to its ill-posedness. To avoid an inverse crime we consider the direct method (Betti's formula) for the inverse problem and we keep the indirect approach as proposed in \cite{Mar90} for the direct problem. This method was introduced in \cite{KreRun05} and then applied in many different problems, see for instance \cite{ChaIvaPro13, IvaJoh07, IvaJoh08, LiSun15, QinCak11} for some recent applications. The system consists of four equations, two on the unknown boundary taking advantage of the boundary conditions and two on the unit circle assuming that we know the far-field pattern of the scattered fields for one or more incident waves.

Even though the first two equations are well-posed because of the equivalence to the system of integral equations for the direct problem, the last two inherit the ill-posedness of the system due to the smoothness of the far-field operators. Following \cite{AltKre12a,JohSle07} we apply a two-step method meaning, we first solve the well-posed subsystem to obtain the corresponding densities and then we solve the linearized (with respect to the boundary) ill-posed subsystem to update the initial approximation of the radial function. We consider Tikhonov regularization and the normal equations are solved by the conjugate gradient method.

The paper is organized as follows: in \autoref{formulation} we formulate the problem in two dimensions and in \autoref{direct} we present the direct scattering problem, the elastic potential and the equivalent system of integral equations. The inverse problem is stated in \autoref{inverse} where we construct an equivalent system of integral equation using the direct method. In \autoref{method} the two-step method for the parametrized form of the system  and the necessary Fr\'echet derivatives of the operators are presented. In the last section, the numerical examples give satisfactory results and demonstrate the applicability of the proposed method.

\section{Problem formulation}\label{formulation}
We consider the scattering of time-harmonic elastic waves by an isotropic and homogeneous elastic inclusion $D_i \subset \R ^2$ with smooth boundary $\Gamma$ described by the Lam\'e parameters $\lambda_i,$ $\mu_i$ and the constant density $\rho_i$. The exterior of $D_i$ described by $D_e = \R^2 \setminus \overline{D}_i$ is filled with an isotropic and homogeneous elastic medium with Lam\'e constants $\lambda_e,$ $\mu_e$ and density $\rho_e$. Henceforth, $j=i,e$ counts for the interior $D_i$ and the exterior domain $D_e ,$ respectively. In addition, we assume that $\lambda_j +\mu_j >0, \, \mu_j >0$ and $\rho_j >0.$

By $\ta$ we define the unit tangent vector to $\Gamma$ and by $\n =\b Q \cdot\ta$ the unit normal vector directed on $D_e$, where $\b Q$ denotes the unitary matrix
\begin{equation*}
\b Q= \begin{pmatrix}
\phantom{-}0 & 1\\ -1 & 0
\end{pmatrix}.
\end{equation*}

The incident field is either a longitudinal plane wave
\[
\b u^{inc}_p (\b x ;  \di) = \di \,  e^{\i k_{p,e} \di \cdot \b x} ,
\]
or a transversal plane wave
\[
\b u^{inc}_s (\b x ;  \di) = -\b Q \cdot \di \,  e^{\i k_{s,e} \di \cdot \b x} ,
\]
where $\di$ is the propagation vector and the wavenumbers are given by
\[
k_{p,j}^2:=\frac{\rho_j \omega^2}{\lambda_j+2\mu_j},\quad k_{s,j}^2:=\frac{\rho_j \omega^2}{\mu_j},
\]
where $\omega >0$ is the circular frequency. In the following, $\alpha = p,s$ counts for the longitudinal and the transversal waves, respectively.

The scattering of $\b u^{inc}$ by the inclusion generates the scattered field $\b u^e , \b x \in D_e$ and the transmitted field  $\b u^i , \b x \in D_i .$ Both of them satisfy the Navier equation in their domains of definition
\begin{equation}\label{eqNavier}
\bm \Delta^\ast_j \b u^j + \rho_j \omega^2 \b u^j = \b 0, \quad \b x \in D_j ,
\end{equation}
with the Lam\'e operator defined by $\bm \Delta^\ast_j := \mu_j \bm \Delta_j + (\lambda_j + \mu_j) \bm \nabla \bm \nabla \cdot .$ If $\b u^j$ satisfies  \eqref{eqNavier}, due to the Helmholtz decomposition, it can be written as a sum of a longitudinal and a transversal wave
\[
\b u^j = \b u^j_p + \b u^j_s , 
\]
which are defined by
\[
\b u^j_p := -\frac1{k^2_{p ,j}} \bm \nabla \bm \nabla \cdot \b u^j, \quad \b u^j_s := \b u^j -\b u^j_p .
\]

On the boundary we impose transmission conditions of the form
\begin{subequations}\label{eqTrans}
\begin{alignat}{2}
\b u^i &= \b u^e + \b u^{inc},  & \quad & \mbox{on  }  \Gamma ,\\
\b T^i  \b u^i &= \b T^e ( \b u^e + \b u^{inc} ), & & \mbox{on  }   \Gamma,
\end{alignat}
\end{subequations}
where the boundary traction operator $\b T^j$ is given by
\begin{equation*}%\label{eqTraction}
\b T^j \b u^j :=\lambda_j \n (\bm \nabla\cdot \b u^j ) +2\mu_j \left( \n \cdot \bm \nabla \right)\b u^j  +\mu_j (\b Q \cdot \n ) \bm \nabla\cdot\left(
\b Q \cdot \b u^j \right) .
\end{equation*}
The field $\b u^e$ is required to satisfy also the Kupradze radiation condition
\begin{equation}\label{eqRadiation}
\lim_{r\rightarrow\infty} \sqrt{r}\left(\frac{\partial \b u^e_\alpha}{\partial r}-\i k_{\alpha ,e} \b u^e_\alpha \right)=\b 0, \quad  r=\left|\b x\right| ,
\end{equation}
uniformly in all directions. Then, the direct elastic scattering problem reads: Given $D_j$ (geometry and elastic parameters) and the incident field $\b u^{inc}_\alpha ,$ solve the boundary value problem \eqref{eqNavier} - \eqref{eqRadiation} to obtain $\b u^j .$

At this point we recall that any solution of \eqref{eqNavier} satisfying \eqref{eqRadiation} has an asymptotic behaviour of the form
\begin{equation*}%\label{eqAsymptotic}
\b u^e_\alpha =\frac{e^{ik_{\alpha ,e} r}}{\sqrt{r}}\left\{\b u_\alpha^\infty (\xhat )+\mathcal O \left(\frac1{r}\right)\right\}, \quad r \rightarrow \infty, 
\end{equation*}
uniformly in all directions $\xhat =\b x/r \in S,$ where $S$ denotes the unit circle. The pair $(\b u_p^\infty , \b u_s^\infty)$ is called the far-field patterns of the scattered field $\b u^e .$

\section{The direct elastic scattering problem}\label{direct}

To represent the solution of the direct and the inverse problem as a combination of an elastic single- and a double-layer potential we first introduce the fundamental solution of the Navier equation
\begin{align*}%\label{eqGreen}
\bm \Phi_j (\b x,\b y) =  \frac{\i}{4\mu_j}H_0^{(1)}(k_{s,j}\left|\b x-\b y\right|)\b I + \frac{\i}{4\rho_j \omega^2}\bm \nabla\bm \nabla^\top \left[H_0^{(1)}(k_{s,j}\left|\b x-\b y\right|)-H_0^{(1)}(k_{p,j}\left|\b x-\b y\right|)\right]
\end{align*}
in terms of the the identity matrix $\b I$ and the Hankel function $H_0^{(1)}$ of order zero and of the first kind. The Green's tensor can be transformed into
\begin{equation*}%\label{eqGreen2}
\bm \Phi_j (\b x,\b y)=\Phi_{1,j} (\left|\b x-\b y\right|)\b I+\Phi_{2,j} (|\b x - \b y|) \b J(\b x -\b y),
\end{equation*}
where the functions $\Phi_{1,j},\, \Phi_{2,j}: \R \rightarrow \C$ are given by \cite{Kre96}
\begin{subequations}\label{eqPhi1}
\begin{alignat}{2}
\Phi_{1,j} (t) &= \frac{\i}{4\mu_j}H_0^{(1)}(k_{s,j} t)-\frac{\i}{4\rho_j \omega^2 t}\left[k_{s,j} H_1^{(1)}(k_{s,j} t)-k_{p,j} H_1^{(1)}(k_{p,j} t)\right] ,  \\
\Phi_{2,j} (t) &= \frac{\i}{4\rho_j \omega^2}\left[\frac{2k_{s,j}}{t}H_1^{(1)}(k_{s,j} t)-k_{s,j}^2 H_0^{(1)}(k_{s,j} t)-\frac{2k_{p,j}}t H_1^{(1)}(k_{p,j} t)+k_{p,j}^2 H_0^{(1)}(k_{p,j} t)\right]
\end{alignat}
\end{subequations}
with $H_1^{(1)} = - H_0^{(1)'}$ and 
\begin{equation*}
\b J(\b x)=\frac{\b x \b x^\top}{|\b x |^2}, \quad \b x\neq \b 0
\end{equation*}
in terms of a dyadic product of $\b x$ with its transpose $\b x^\top$.  Then, for the vector density $\bm \varphi \in [C^{0,a}(\Gamma) ]^2$, $0<a \leq 1,$ we introduce the elastic single-layer potential
\begin{equation}\label{eqSingle}
(\b S_j \bm \varphi)(\b x) =\int_{\Gamma} \bm \Phi _j (\b x,\b y)\cdot \bm \varphi(\b y)ds(\b y), \quad \b x \in D _j \backslash \Gamma ,
\end{equation}
and the elastic double-layer potential 
\begin{equation}\label{eqDouble}
(\b D_j \bm\varphi)(\b x) =\int_{\Gamma}\left[\b T^j_y \bm \Phi _j (\b x,\b y) \right]^\top\cdot \bm \varphi(\b y)ds(\b y), \quad \b x \in D _j \backslash \Gamma .
\end{equation}
It is well known that $\b S_j$ and $\b T^j_x \b D_j$ are continuous in $\R^2$ but both $\b D_j$ and $\b T^j_x \b S_j$ satisfy the following jump relations \cite{Kup79}
\begin{subequations}\label{eqJump}
\begin{alignat}{2}
\b D_j\bm \varphi &= \left(\pm \frac12 \b I +\b K_j\right)\bm \varphi,  & \quad & \mbox{on  }  \Gamma  ,\\
\b T^j_x \b S_j \bm \varphi &= \left(\mp \frac12 \b I +\b L_j\right)\bm \varphi,  & \quad & \mbox{on  }  \Gamma , \\
\b T^j_x \b D_j\bm \varphi &= \b N_j \bm \varphi, & \quad & \mbox{on  }  \Gamma ,
\end{alignat}
\end{subequations}
where the upper (lower) sign corresponds to the limit $\b x\rightarrow \Gamma$ from $D_e$ ($D_i$), and the integral operators are defined by
\begin{subequations}\label{eqTraceIntegral}
\begin{alignat}{2}
(\b K_j\bm \varphi)(\b x) &=\int_{\Gamma}\left[\b T^j_y \bm \Phi _j (
\b x,\b y) \right]^\top\cdot \bm\varphi(\b y)ds(\b y),  & \quad & \b x \in \Gamma ,\\
(\b L_j \bm \varphi)(\b x) &=\int_{\Gamma} \b T^j_x \bm\Phi _j (\b x,\b y) \cdot \bm \varphi(\b y)ds(\b y),  & \quad & \b x \in \Gamma , \\
(\b N_j\bm \varphi)(\b x) &= \b T^j_x \int_{\Gamma}\left[\b T^j_y \bm \Phi _j (
\b x,\b y) \right]^\top\cdot \bm\varphi(\b y)ds(\b y),  & \quad & \b x \in \Gamma .
\end{alignat}
\end{subequations}
All the above integrals are well defined and in particular the operator $\b S_j$ for $\b x \in \Gamma$ is weakly singular, the operators $\b K_j , \, \b L_j$ are singular and $\b N_j$ admits a hypersingular kernel. From the asymptotic behaviour of the Hankel functions we can compute also the far-field patterns of the single- \eqref{eqSingle} and double-layer potential \eqref{eqDouble} \cite{ChaKreMon00, Kre96}
 \begin{subequations}\label{eqFarPotentials}
\begin{alignat}{2}
(\b S^\infty_\alpha \bm \varphi)(\xhat) &=\beta_\alpha \int_{\Gamma} \b J_\alpha (\xhat) \cdot \bm \varphi(\b y) \, e^{-\i k_{\alpha , e} \xhat \cdot \b y}ds(\b y), & \quad & \xhat \in S,\\
(\b D^\infty_\alpha \bm\varphi)(\xhat) &=\gamma_\alpha \int_{\Gamma}\b J_\alpha (\xhat) \cdot \b F (\xhat , \b y) \cdot\bm \varphi(\b y) \, e^{-\i k_{\alpha , e} \xhat \cdot \b y}ds(\b y), & \quad & \xhat \in S,
\end{alignat}
\end{subequations}
with the coefficients
 \begin{equation*}
\begin{aligned}
\beta_p &= \frac{e^{\i\pi/4}}{\lambda_e +2\mu_e}\frac1{\sqrt{8\pi k_{p,e}}}, & \beta_s &= \frac{e^{\i\pi/4}}{\mu_e}\frac1{\sqrt{8\pi k_{s,e}}}, \\
\gamma_p &= \frac{e^{-\i\pi/4}}{\lambda_e +2\mu_e}\sqrt{ \frac{ k_{p,e}}{8\pi }}, & \gamma_s &= \frac{e^{-\i\pi/4}}{\mu_e}\sqrt{ \frac{ k_{s,e}}{8\pi }}, \\
\end{aligned}
\end{equation*}
and the matrices $\b J_p (\xhat) = \b J (\xhat), \,  \b J_s (\xhat) = \b I - \b J (\xhat)$ and
\[ 
\b F (\xhat , \b y) = \lambda_e \xhat \,\n (\b y)^\top + \mu_e \n (\b y) \,\xhat^\top + \mu_e  (\n (\b y) \cdot \xhat ) \b I . 
\]

Considering the indirect integral equation method, we search the solution of the direct scattering problem in the form 
\begin{equation}\label{eqFieldDire}
\b u^j(\b x)=(\b D_j \bm \varphi _j)(\b x)+(\b S_j \bm \psi _j)(\b x), \quad \b x \in D_j .
\end{equation}
To simplify the above representation, we set 
\[
\bm \varphi _j (\b x) = \tau_j \bm \varphi (\b x) , \quad \bm \psi_j (\b x) = \bm \psi (\b x), \quad \tau_j =\frac{\lambda_{j} +2\mu_{j}}{\mu_{j} (\lambda_{j} +\mu_{j})}
\]
and the formula \eqref{eqFieldDire} is reduced to
\begin{equation}\label{eqFieldDire2}
\b u^j(\b x)=\tau_j (\b D_j \bm \varphi)(\b x)+(\b S_j \bm \psi )(\b x), \quad \b x \in D_j .
\end{equation}
Using this representation, applying the boundary conditions \eqref{eqTrans} and the jump relations \eqref{eqJump} we see that the densities $\bm \varphi , \,\bm \psi$ satisfy the system of integral equations 
\begin{equation}\label{eqSystemDire}
\begin{pmatrix}
\b I +\b L_i -\b L_e & \tau_i \b N_i -\tau_e \b N_e \\
\b S_i -\b S_e &  -\frac{\tau_i +\tau_e}{2}\b I+\tau_i \b K_i -\tau_e \b K_e
\end{pmatrix}
\begin{pmatrix}
\bm \psi \\ \bm \varphi  
\end{pmatrix} =
\begin{pmatrix}
\b T^e \b u^{inc} |_\Gamma \\
\b u^{inc} |_\Gamma
\end{pmatrix} .
\end{equation}
The following result regarding uniqueness and existence was proved in \cite{Mar90}.
\begin{theorem}\label{th1}
The system of integral equations \eqref{eqSystemDire} has precisely one solution $(\bm \varphi , \bm \psi)$, with $\bm \varphi \in [C^{1,a}(\Gamma)]^2$ and $\bm \psi \in [C^{0,a} (\Gamma)]^2$. Moreover, the corresponding displacement fields \eqref{eqFieldDire2} solve the direct scattering problem \eqref{eqNavier} - \eqref{eqRadiation}.
\end{theorem}
Then, the solution of the direct problem \eqref{eqFieldDire2} provides us with the far-field pattern $(\b u_p^\infty , \b u_s^\infty)$ given by
\begin{equation}\label{eqFar}
\b u_\alpha^\infty (\xhat) = \int_\Gamma \b J_\alpha (\xhat) \cdot  \left[ 
\tau_e \gamma_\alpha \b F (\xhat , \b y) \cdot\bm \varphi(\b y) + \beta_\alpha  \bm \psi (\b y)\right] e^{-\i k_{\alpha , e} \xhat \cdot \b y}ds(\b y),  \quad \xhat \in S,
\end{equation}
where we have used the asymptotic forms \eqref{eqFarPotentials} and that $\bm \varphi ,\, \bm \psi$ are the solutions of \eqref{eqSystemDire}. 

\begin{remark}\label{remark1}
The choice of $\tau_j$ is not random since, as we are going to see later, the combination $\tau_i \b N_i -\tau_e \b N_e$ turns out to be a weakly singular operator, reducing the hypersingularity of $\b N_j. $
\end{remark}

For the numerical implementation, we present also different representations of the solutions in order to distinguish from the formulas we will derive later for the solution of the inverse problem, even though, we are going to consider the direct method. We do not consider the solvability of the above systems. Let the solution of the direct problem be given by
\begin{equation}\label{eqSingleSol}
\b u^j (\b x) =  (\b S_j \bm \psi_j ) (\b x), \quad \b x \in D_j .
\end{equation}
Then, the densities satisfy the system of equations
\begin{equation*}
\begin{pmatrix}
\b S_i  & - \b S_e \\
\frac12 \b I +\b L_i & \frac12 \b I -\b L_e  
\end{pmatrix}
\begin{pmatrix}
\bm \psi_i \\ \bm \psi_e 
\end{pmatrix} =
\begin{pmatrix}
\b u^{inc} |_\Gamma \\
\b T^e \b u^{inc} |_\Gamma
\end{pmatrix} 
\end{equation*}
and we obtain the far-field patterns $\b u_\alpha^\infty (\xhat ) = (\b S_\alpha^\infty \bm \psi_e ) (\xhat ).$ If we consider the representations 
\begin{equation}\label{eqDoubleSol}
\b u^j (\b x)=  (\b D_j \bm \psi_j )(\b x),\quad \b x \in D_j ,
\end{equation}
the densities satisfy
\begin{equation*}
\begin{pmatrix}
-\frac12 \b I +\b K_i & -\frac12 \b I -\b K_e   \\
\b N_i  & - \b N_e 
\end{pmatrix}
\begin{pmatrix}
\bm \psi_i \\ \bm \psi_e 
\end{pmatrix} =
\begin{pmatrix}
\b u^{inc} |_\Gamma \\
\b T^e \b u^{inc} |_\Gamma
\end{pmatrix},
\end{equation*}
resulting to the far-field patterns $\b u_\alpha^\infty (\xhat ) = (\b D_\alpha^\infty \bm \psi_e ) (\xhat ).$

\section{The inverse elastic scattering problem}\label{inverse}
Now we can state the inverse problem, which reads: Find the  shape and the position of the inclusion $D_i$ (i.e. reconstruct the boundary) from the knowledge of the far-field patterns $(\b u_p^\infty (\xhat), \b u_s^\infty (\xhat))$ for all $\xhat \in S,$ for one incident plane wave $\b u^{inc}_\alpha$ with direction $\di$ either longitudinal or transverse. In general, the unique solvability of the inverse problem for one or even for a finite number of incident waves is an open problem. Uniqueness for the transmission problem exists only for infinitely many incident waves \cite{HahHsi93}. There exist also results for a rigid scatterer, local uniqueness in $\R^2$ \cite{GinMin11} and measuring only $\b u_s^\infty$  for a transversal incident plane wave and simple geometries in $\R^3$ \cite{HuKirSin13}.

\subsection{The integral equation method}

To solve numerically this problem, we consider the non-linear integral equation method, introduced in \cite{KreRun05}, but here we apply the direct method in contrast to the forward problem. We recall for $\b v, \b w \in [C^2 (\overline{D}_j)]^2$ the third Betti's formula
\[
\int_{D_j} (\b v \cdot \bm \Delta_j^\ast \b w - \b w \cdot \bm \Delta_j^\ast \b v) d\b x = \int_\Gamma (\b v \cdot \b T^j \b w - \b w \cdot \b T^j \b v) ds (\b x).
\]
Using the definitions \eqref{eqSingle} and \eqref{eqDouble}, we consider the above formula once for the field $\b u^e$ and the tensor $\bm \Phi_e$ in $D_e ,$ and then for $\b u^{inc}, \, \bm \Phi_e$ in $D_e$ to obtain
\begin{subequations}
\begin{alignat}{2}
\b u^e (\b x) &= (\b D_e \b u^e ) (\b x) - (\b S_e ( \b T^e \b u^e ) )(\b x), & \quad & \b x\in D_e , \label{eqBetti1} \\
\b 0 &= (\b D_e \b u^{inc} ) (\b x) - (\b S_e ( \b T^e \b u^{inc} ) )(\b x), & \quad & \b x\in D_e .\label{eqBetti2}
\end{alignat}
\end{subequations}
We define $\b u^t :=\b u^e +\b u^{inc}$ and by adding \eqref{eqBetti1} and \eqref{eqBetti2} we obtain
\begin{equation}\label{eqBetti3}
\b u^e (\b x) = (\b D_e \b u^t ) (\b x) - (\b S_e ( \b T^e \b u^t ) )(\b x), \quad \b x\in D_e .
\end{equation}
Similarly, for $\b u^i ,\, \bm \Phi_i$ in $D_i ,$ the third Betti's formula results to
\begin{align}
 -\b u^i (\b x) = (\b D_i \b u^i ) (\b x) - (\b S_i ( \b T^i \b u^i ) )(\b x) = (\b D_i \b u^t ) (\b x) - (\b S_i ( \b T^e \b u^t ) )(\b x), \quad \b x\in D_i , \label{eqBetti4}
 \end{align}
where for the last equality we have used the transmission conditions \eqref{eqTrans}. We set $\ka = \b u^t |_\Gamma$ and  $\mi = \b T^e\b u^t |_\Gamma$ and letting $\b x\rightarrow \Gamma$ in the above representations, taking the traction and considering the jump relations \eqref{eqJump}, we get
 \begin{subequations}
\begin{alignat*}{2}
\b u^j (\b x) &= (\tfrac12 \b I \pm \b K_j )\ka (\b x) \mp (\b S_j \mi )(\b x), & \quad & \b x\in \Gamma , \\
\b T^j\b u^j (\b x) &= \pm (\b N_j \ka )(\b x) + (\tfrac12 \b I \mp \b L_j )\mi (\b x), & \quad & \b x\in \Gamma . 
\end{alignat*}
\end{subequations}
We consider \eqref{eqTrans} to obtain
 \begin{subequations}\label{eqInverse1}
\begin{alignat}{2}
 (\tfrac12 \b I - \b K_e )\ka  + \b S_e \mi  &= \b u^{inc} , & \quad & \mbox{on  }  \Gamma , \label{eqTrans1} \\
 (\tfrac12 \b I + \b K_i )\ka  - \b S_i \mi  &= \b 0 , & \quad & \mbox{on  }  \Gamma ,\label{eqTrans2} \\ 
- \b N_e \ka  + (\tfrac12 \b I + \b L_e )\mi  &= \b T^e \b u^{inc} , & \quad & \mbox{on  }  \Gamma ,\label{eqTrans3} \\
\b N_i \ka  + (\tfrac12 \b I - \b L_i )\mi  &= \b 0 , & \quad & \mbox{on  }  \Gamma .\label{eqTrans4}
\end{alignat}
\end{subequations}
In addition, given the far-field operators \eqref{eqFarPotentials} and the representation \eqref{eqBetti3} of the exterior field we observe that the unknown boundary $\Gamma$ and the densities satisfy the (far-field) equation
\begin{equation*}
\begin{pmatrix}
\b D^\infty_p \vspace{0.1cm}\\ 
\b D^\infty_s
\end{pmatrix} \ka  - \begin{pmatrix}
\b S^\infty_p \vspace{0.1cm}\\ 
\b S^\infty_s
\end{pmatrix} \mi = \begin{pmatrix}
\b u^\infty_p \vspace{0.1cm}\\ 
\b u^\infty_s
\end{pmatrix}, \quad \mbox{on  }   S,
\end{equation*}
or in compact form
\begin{equation}\label{eqFarInv}
\bb{D}^\infty \ka  -\bb{S}^\infty \mi =\bb{U}^\infty
\end{equation}
where the right-hand side is the known far-field patterns from the direct problem. We observe that we have six equations \eqref{eqInverse1} and \eqref{eqFarInv} for the three unknowns $\Gamma , \ka$ and $\mi.$ In order to take advantage of the well-posedness of the direct problem, we consider the linear combinations  \eqref{eqTrans1} +  \eqref{eqTrans2} and  $\tau_e \cdot$\eqref{eqTrans3} +  $\tau_i \cdot$\eqref{eqTrans4} for the equations on the boundary and we keep the overdetermined far-field equation. 
% Later we will investigate numerically if the choice of one type of far-field patterns does not affect the reconstructions, as discussed in similar problems \cites{GinSinTha12}. 
Then, we can state the following theorem as a formal formulation of the inverse problem.

\begin{theorem} Given an incident field $\b u_\alpha^{inc}, \, \alpha = p$ or $s$ and the far-field patterns $\bm{\mathcal{U}}^\infty,$ for all $\xhat \in S,$ if $\Gamma$ and the vector densities $\ka , \mi$ satisfy the system of integral equations
 \begin{subequations}\label{eqInverseFinal}
\begin{alignat}{2}
 ( \b I + \b K_i - \b K_e )\ka  + (\b S_e - \b S_i )\mi  &= \b u^{inc} |_\Gamma, \label{eqInverseFinal1} \\
(\tau_i \b N_i- \tau_e \b N_e ) \ka  + (\tfrac{\tau_i + \tau_e}2 \b I + \tau_e\b L_e -\tau_i\b L_i )\mi  &= \tau_e \b T^e \b u^{inc}|_\Gamma, \label{eqInverseFinal2} \\
\bb{D}^\infty \ka  -\bb{S}^\infty \mi &=\bb{U}^\infty , \label{eqInverseFinal3}
\end{alignat}
\end{subequations}
then, $\Gamma$ solves the inverse problem.
\end{theorem}
The integral operators involved in \eqref{eqInverseFinal} are linear with respect to the densities but non-linear with respect to the boundary $\Gamma.$ The subsystem \eqref{eqInverseFinal1} - \eqref{eqInverseFinal2} is equivalent to \eqref{eqSystemDire}, thus well-posed as already proved \cite{Mar90}. The ill-posedness of the inverse problem is then due to the smooth kernels of the far-field operators in \eqref{eqInverseFinal3}.

In general, there exist three different iterative methods to solve the system \eqref{eqInverseFinal} by linearization:
\begin{enumerate}
\item[A.] Given initial guesses for the boundary and the densities, we linearize all three equations in order to update all the unknowns.
\item[B.] Given initial guess for the boundary, we solve the subsystem \eqref{eqInverseFinal1} - \eqref{eqInverseFinal2} to obtain the densities. Then, keeping the densities fixed we solve the linearized equation \eqref{eqInverseFinal3} to obtain the update for the boundary.
\item[C.] Given initial guesses for the densities, we solve the far-field equation 
\eqref{eqInverseFinal3} to obtain $\Gamma$ and then we solve the linearized form
of \eqref{eqInverseFinal1} - \eqref{eqInverseFinal2} to obtain the densities. 
\end{enumerate}
The linearization, using Fr\'echet derivatives of the operators, and the regularization of the ill-posed equations are needed in all methods. However, the iterative method A requires the calculation of the Fr\'echet derivatives of the operators with respect to all the unknowns and the selection of two regularization parameters at every step. Thus, we prefer to use one of the so-called two-step methods B or C. Between the two methods, it is obvious that the second method is preferable since we solve first a well-posed linear system and then we linearize only the far-field operators (operators with smooth and simple kernels). 
From now on, we focus on Method B, a method introduced in \cite{JohSle07} and then applied in different problems, see for instance \cite{AltKre12a, Lee15} for some recent applications.

\section{The two-step method}\label{method}
To analyse further the Method B, we consider the following parametrization for the boundary
\[
\Gamma = \{ \b z (t) =  r (t) (\cos t,\, \sin t) : t \in [0,2\pi]\},
\]
where $\b z : \R \rightarrow \R^2$ is a $C^2$-smooth, $2\pi$-periodic parametrization. We assume in addition that $\b z$  is injective in $[0,2\pi),$ that is $\b z'  (t) \neq 0,$ for all $t\in [0,2\pi].$ The non-negative function $r$ represents the radial distance of $\Gamma$ from the origin. Then, we define
\[
\ksi (t) = \ka (\b z (t)), \quad    \zita (t) = \mi (\b z (t)), \quad t \in [0,2\pi]
\]
and the parametrized form of \eqref{eqInverseFinal} is given by
\begin{equation}\label{eqFinal}
\bb A (r; \ksi) + \bb B (r; \zita ) = \bb C , 
\end{equation}
where
\begin{equation*}
\bb{A}=\begin{pmatrix}
\bb{A}_1  \\
\bb{A}_2\\
\bb{A}_3 \end{pmatrix}, \quad \bb{B}=\begin{pmatrix}
\bb{B}_1  \\
\bb{B}_2\\
\bb{B}_3 \end{pmatrix}, \quad \bb{C}=\begin{pmatrix}
\bb{C}_1   \\
\bb{C}_2 \\
\bb{C}_3 \end{pmatrix},
\end{equation*}
with the parametrized operators
\begin{equation*}
\begin{aligned}
(\bb{A}_1 (r; \ksi))(t) &= \ksi (t)+\int_0^{2\pi} \left[\b T^i_{\b z(\tau)} \bm \Phi _i (t,\tau)- \b T^e_{\b z(\tau)} \bm \Phi _e (t,\tau) \right]^\top  \cdot \ksi(\tau) |\b z'  (\tau)| d\tau, \\
(\bb{A}_2 (r; \ksi))(t) &=   \int_0^{2\pi} \left( \tau_i  \b T^i_{\b z(t)} \left[\b T^{i}_{\b z (\tau)} \bm \Phi _i (t,\tau) \right]^\top - \tau_e \b T^{e}_{\b z(t)} \left[\b T^{e}_{\b z(\tau)} \bm \Phi _{e} (t,\tau) \right]^\top \right)  \cdot \ksi(\tau) |\b z'  (\tau)| d\tau, \\
(\bb{A}_3 (r; \ksi))(t) &=  (\bb D^\infty (r; \ksi))(t), \\
(\bb{B}_1 (r; \zita ))(t) &=\int_0^{2\pi} \left[ \bm\Phi _{e} (t,\tau)-\bm \Phi _{i}(t,\tau) \right] \cdot \zita(\tau )  |\b z'  (\tau)| d\tau, \\
(\bb{B}_2 (r; \zita ))(t) &= \frac{\tau_i +\tau_e}{2} \zita (t) + \int_0^{2\pi}\left(  \tau_e \b T^{e}_{\b z(t)} \bm\Phi _e (t,\tau)  - \tau_i  \b T^i_{\b z(t)} \bm \Phi_i (t,\tau) \right) \cdot  \zita(\tau)|\b z'  (\tau)|  d\tau , \\
 (\bb{B}_3 (r; \zita ))(t) &=  -(\bb S^\infty (r; \zita ))(t),
\end{aligned}
\end{equation*}
where $\bm \Phi _j (t,\tau) := \bm \Phi _j (\b z(t),\b z(\tau)),$
\begin{align*}
(\b D^\infty_\alpha (r; \ksi))(t) &=\gamma_\alpha \int_0^{2\pi} \b J_\alpha (\xhat (t)) \cdot \b F (\xhat(t) , \b z (\tau )) \cdot\bm \ksi(\tau) \, e^{-\i k_{\alpha , e} \xhat (t) \cdot \b z (\tau)} |\b z'  (\tau)| d\tau, \\
(\b S^\infty_\alpha (r; \zita ))(t) &=\beta_\alpha \int_0^{2\pi} \b J_\alpha (\xhat (t)) \cdot \zita (\tau) \, e^{-\i k_{\alpha , e} \xhat (t) \cdot \b z (\tau)} |\b z'  (\tau)| d\tau
\end{align*}
and the right-hand side
\begin{equation*}
(\bb{C}_1 (r))(t) = \b u^{inc} (\b z(t)), \quad   (\bb{C}_2 (r)) (t) = \tau_e (\b T^{e}\b u^{inc}) (\b z(t)), \quad \bb{C}_3  (t) = \bb{U}^\infty (\xhat (t)).
\end{equation*}
\begin{remark}
The operators $\bb{A}_k ,\, \bb{B}_k ,\, k=1,2,3$ act on the densities and the first variable $r$ shows the dependence on the unknown parametrization of the boundary. Only $\bb{C}_3$ is independent of the radial function. 
\end{remark}

The two-step method for the system \eqref{eqFinal} reads as follows:

\begin{assumption}\label{IterationScheme}
Initially, we give an approximation of the radial function $r^{(0)}$. Then, in the $k$th iteration step:
\begin{enumerate}[i.]
\item We assume that we know $r^{(k-1)}$ and we solve the subsystem
\begin{equation}\label{sub_eq}
\begin{pmatrix}
\bb{A}_1  \\
\bb{A}_2 \end{pmatrix} (r^{(k-1)}; \ksi) + \begin{pmatrix}
\bb{B}_1  \\
\bb{B}_2 \end{pmatrix}(r^{(k-1)}; \zita) = \begin{pmatrix}
\bb{C}_1   \\
\bb{C}_2  \end{pmatrix} (r^{(k-1)}),
\end{equation}
to obtain the densities $\ksi^{(k)}, \, \zita^{(k)}. $
\item Then, keeping the densities fixed, we linearize the third equation of \eqref{eqFinal}, namely
\begin{equation}\label{linear_eq}
\bb A_3 (r^{(k-1)}; \ksi^{(k)}) + (\bb A'_3 (r^{(k-1)}; \ksi^{(k)})) (q) + \bb B_3 (r^{(k-1)}; \zita^{(k)} )  + (\bb B'_3 (r^{(k-1)}; \zita^{(k)} )(q)= \bb C_3 . 
\end{equation}
We solve this equation for $q$ and we update the radial function $r^{(k)} = r^{(k-1)} +q.$
\end{enumerate}
The iteration stops when a suitable stopping criterion is satisfied.
\end{assumption}
The function $q$ stands for the radial function of the perturbed boundary
\[
\Gamma_q = \{ \b q (t) =  q (t) (\cos t,\, \sin t) : t \in [0,2\pi]\},
\]
and the Fr\'echet derivatives of the operators are calculated by formally differentiating their kernels with respect to $r$ \cite{Cha95}
\[
(\bb A'_3 (r ; \ksi )) (q) =  \begin{pmatrix}
\left.(\b D^{\infty}_p \right.' (r; \ksi))(q) \\
\left.(\b D^\infty_s\right.' (r; \ksi))(q)
\end{pmatrix}, \quad
(\bb B'_3 (r ; \zita )) (q) = -\begin{pmatrix}
\left.(\b S^{\infty}_p\right.' (r; \zita))(q) \\
\left.(\b S^\infty_s\right.' (r; \zita))(q)
\end{pmatrix}
\]
with
\begin{align*}
(\left.(\b D^{\infty}_\alpha \right.' (r; \ksi))(q))(t) &= \gamma_\alpha \int_0^{2\pi} \b J_\alpha (\xhat (t)) \cdot \b G_\alpha (\xhat(t) , \b z (\tau ),\b q (\tau )) \cdot\bm \ksi(\tau) \, e^{-\i k_{\alpha , e} \xhat (t) \cdot \b z (\tau)}d\tau,  \\
(\left.(\b S^{\infty}_\alpha \right.' (r; \zita ))(q))(t) &= \beta_\alpha \int_0^{2\pi}  g_\alpha (\xhat(t) , \b z (\tau ),\b q (\tau )) \b J_\alpha (\xhat (t)) \cdot \zita (\tau) \, e^{-\i k_{\alpha , e} \xhat (t) \cdot \b z (\tau)}d\tau ,
\end{align*}
where
\begin{align*}
\b G_\alpha (\xhat(t) , \b z (\tau ),\b q (\tau )) &= \lambda_e  \xhat (t) \b v (\tau)^\top+\mu_e \b v (\tau) \xhat (t)^\top +\mu_e (\b v (\tau)\cdot \xhat(t) ) \b I \\
&\phantom{=} -\i k_{\alpha , e} (\xhat (t) \cdot \b q (\tau) )\left|\b z' (\tau)\right| \b F (\xhat(t) , \b z (\tau )) 
\end{align*}
for $\b v(\tau) := \b Q \cdot \b q'(\tau)$ and 
\[
g_\alpha (\xhat(t) , \b z (\tau ),\b q (\tau )) = -\i k_{\alpha , e} (\xhat (t) \cdot \b q (\tau) )\left|\b z' (\tau)\right| + \frac{\b z'(\tau) \cdot \b q' (\tau) }{\left|\b z' (\tau)\right|} .
\]

To show injectivity of the integral operators involved in \eqref{linear_eq}, we consider a simplified linearization. Assuming that $\b z'$ is known, we linearize with respect to $\b z$ only, viewing $\b z'$ as independent of $\b z,$ resulting to
 \begin{align*}
(\left.(\b D^{\infty}_\alpha \right.' (r; \ksi))(q))(t) &= -\i k_{\alpha , e} \gamma_\alpha \int_0^{2\pi}  (\xhat (t) \cdot \b q (\tau) ) \b J_\alpha (\xhat (t)) \cdot \b F (\xhat(t) , \b z (\tau )) \cdot\bm \ksi(\tau)   e^{-\i k_{\alpha , e} \xhat (t) \cdot \b z (\tau)}d\tau,  \\
(\left.(\b S^{\infty}_\alpha \right.' (r; \zita ))(q))(t) &= -\i k_{\alpha , e} \beta_\alpha \int_0^{2\pi}  (\xhat (t) \cdot \b q (\tau) )\b J_\alpha (\xhat (t)) \cdot \zita (\tau) \, e^{-\i k_{\alpha , e} \xhat (t) \cdot \b z (\tau)}d\tau ,
\end{align*}
where now $\bm \ksi(\tau) := \bm \ksi(\tau) \left|\b z' (\tau)\right|$ and  $\bm \zita (\tau) := \bm \zita(\tau) \left|\b z' (\tau)\right|.$ In addition, we recall that for sufficiently small $q,$ the perturbed boundary $\Gamma_q$ can be represented by 
$\b q (t) = \tilde q (t) \b Q \cdot \b z'(t),$ $t \in [0,2\pi]$ \cite{IvaKre06}. Now we can state the following theorem considering the above formulas for the Fr\'echet derivatives and as unknown the function $\tilde q.$

\begin{theorem}
Let $\bm \ksi, \, \bm \zita$ solve \eqref{sub_eq} and let $r$ be the radial function of the unperturbed boundary $\Gamma$. If $\tilde q \in C^2 [0,2\pi]$ satisfies the homogeneous form of equation  \eqref{linear_eq}, meaning
\begin{equation}\label{homo}
(\bb A'_3 (r; \ksi) + \bb B'_3 (r; \zita ))(\tilde q)= 0,
\end{equation}
then $\tilde q =0.$
\end{theorem}
\begin{proof}
We follow the ideas presented in \cite{IvaKre06} for the Laplace operator. Equation \eqref{homo} is equivalent to
\begin{equation}\label{homo2}
(\left.\b D^{\infty}_\alpha \right.' (r; \ksi) - \left.\b S^{\infty}_\alpha \right.' (r; \zita))(\tilde q)= 0, \quad \alpha = p,s .
\end{equation}
We introduce the function
\begin{align*}
V (\b x) &= \int_0^{2\pi} \frac{\partial}{\partial \n (\b z (\tau)) } \left[\b T^e_{\b z(\tau)} \bm \Phi _e (\b x,\b z(\tau)) \right]^\top  \cdot \ksi (\tau)  \tilde{q} (\tau) \left|\b z' (\tau)\right| d\tau\\
&\phantom{=}-\int_0^{2\pi}  \frac{\partial}{\partial \n (\b z (\tau)) } \bm \Phi _e (\b x,\b z(\tau))\cdot \zita (\tau)  \tilde{q} (\tau) \left|\b z' (\tau)\right| d\tau,\quad \b x \in D_e ,
\end{align*}
that is a radiating solution of \eqref{eqNavier} in $D_e.$ The far-field patterns of $V$ are given by
\begin{align*}
V_\alpha^\infty (\xhat) &= \int_0^{2\pi} \frac{\partial}{\partial \b z (\tau) } \bm \Psi^\infty_\alpha (\xhat,\b z(\tau))  \cdot \ksi (\tau)  \tilde{q} (\tau) \left|\b z' (\tau)\right| d\tau\\
&\phantom{=}-\int_0^{2\pi}  \frac{\partial}{\partial \b z (\tau) }\bm \Phi^\infty_\alpha (\xhat,\b z(\tau))\cdot \zita (\tau)  \tilde{q} (\tau) \left|\b z' (\tau)\right| d\tau,\quad \xhat \in S ,
\end{align*}
where
$
\bm \Psi^\infty_\alpha (\xhat,\b z(\tau)) = \gamma_\alpha \b J_\alpha (\xhat) \cdot \b F (\xhat , \b z(\tau)) \, e^{-\i k_{\alpha , e} \xhat \cdot \b z(\tau)}
$
and $\bm \Phi^\infty_\alpha (\xhat,\b z(\tau)) = \beta_\alpha \b J_\alpha (\xhat)  \, e^{-\i k_{\alpha , e} \xhat \cdot \b z(\tau)}.$
We observe that $V^\infty_\alpha$ coincide with the left hand-side of \eqref{homo2} since also $\b F$ is independent of $\b z.$ Then, $V_\alpha^\infty  \equiv 0,$ and by Rellich's Lemma we get that $V (\b x)= 0, \, \b x \in D_e .$ 
%Approaching $\Gamma$ from $D_e$ we get that \cite{Kup79}
%\begin{multline}
%\int_0^{2\pi} \frac{\partial}{\partial \n (\b z (\tau)) } \left[\b T^e_{\b z(\tau)} \bm \Phi _e (\b x,\b z(\tau)) \right]^T  \cdot \ksi (\tau)  \tilde{q} (\tau) \left|\b z' (\tau)\right| d\tau 
%+\frac1{2  \left|\b z' (t)\right|} ( \n (\b z (t)) \cdot \tilde{\b q} (t) ) \zita (t) \\ -\int_0^{2\pi}  \frac{\partial}{\partial \n (\b z (\tau)) } \bm \Phi _e (\b x,\b z(\tau))\cdot \zita (\tau)  \tilde{q} (\tau) \left|\b z' (\tau)\right| d\tau = 0,\quad \b x \in \Gamma. 
%\end{multline}
In this equation the first integral has a hypersingular kernel and the second one a kernel with lower singularity. Since the fundamental solution of the Navier equation has the same (logarithmic) singularity as the fundamental solution of the Laplace equation, we can show that $\ksi (t) \tilde{q}(t) = 0$ for almost every $t\in [0,2\pi]$ \cite{IvaKre06}. An application of unique continuation and Holmgren's theorem \cite{KnoPay71} results to $\tilde{q}=0,$ since $\ksi$ cannot be zero on $\Gamma.$ 
\end{proof}

\section{Numerical implementation}\label{examples}
In this section we firstly justify numerically the convergence of the proposed scheme using analytic solutions of the direct problem and then we investigate the applicability of the \autoref{IterationScheme} for solving the inverse problem. We solve both integral equations systems using the Nystr\"om method. 

To handle the singularities of the kernels we consider the usual quadrature rules based on trigonometric interpolation \cite{Kre95, Kre99}. For smooth kernels we use the trapezoidal rule. The exact forms of the parametrized kernels are presented in \cite{Cha04, ChaKreMon00}. Thus, here we only briefly present the form of the kernel 
of the operator $\tau_i \b N_i -\tau_e \b N_e$ appearing in \eqref{eqSystemDire} and in $\bb A_2 .$ This combination of operators consists of two hypersingular terms but it turns out to be weakly singular, as discussed in \autoref{remark1}. We consider the following decomposition
\begin{equation}\label{traction_kernel}
\b T^j_{\b z(t)} \left[\b T^j_{\b z (\tau)} \bm \Phi _j (t,\tau) \right]^\top = \b T^j_{\b z(t)} \left[\b T^j_{\b z (\tau)} \left( \bm \Phi _j (t,\tau) -\bm \Phi^{(0)} _j (t,\tau)\right) \right]^\top + \b T^j_{\b z(t)} \left[\b T^j_{\b z (\tau)} \bm \Phi^{(0)}_j (t,\tau) \right]^\top ,
\end{equation}
where $\bm \Phi^{(0)}_j$ denotes the fundamental solution of the static ($\omega =0$) Navier equation. The first term is weakly singular and the second one preserves the hypersingularity. The advantage, of this decomposition, is that the second term coming from the static case is easier to handle by a Maue-type expression \cite{ChaKreMon00}, although it is not needed here. The integral operator with kernel the second term can be written as \cite[Equation 2.6]{Cha04}
\begin{align*}
(\b N^{(0)}_j (r; \ksi))(t)
 &= \int_0^{2\pi} \b T^j_{\b z(t)} \left[\b T^j_{\b z (\tau)} \bm \Phi^{(0)}_j (t,\tau) \right]^\top \cdot\bm \ksi(\tau) |\b z'  (\tau)| d\tau \\
 &= \frac{c_j}{2\pi \left|\b z' (t)\right|} \int_0^{2\pi} \left[\cot \frac{\tau -t}{2} \bm \ksi'(\tau) + \b K (t,\tau) \cdot\bm \ksi(\tau) \right] d\tau ,
 \end{align*}
for a smooth kernel $\b K$ independent of $D_j$ and $c_j = \mu_j (\lambda_j + \mu_j) / (\lambda_j + 2\mu_j) .$ Since $c_j = \tau^{-1}_j ,$ we see that
\[
\tau_i (\b N^{(0)}_i (r; \ksi))(t) - \tau_e (\b N^{(0)}_e (r; \ksi))(t) = 0.
\]
Then, the combination $\tau_i \b N_i -\tau_e \b N_e$ presents only weakly singularity due to the first term in \eqref{traction_kernel}. 

The error and convergence analysis of the proposed numerical method can be carried out based on the theory of operator approximations and on estimates for trigonometrical interpolation in Sobolev spaces \cite[Section 12.4]{Kre99}. This analysis shows that the applied method admits super-algebraic convergence and in the case of analytical data it convergences exponentially.

In the following examples, we consider three different parametrizations of the boundary curves.
%, see \autoref{Fig1}. 
%
%\begin{figure}[t]
%\begin{center}
%\includegraphics[scale=0.55]{figures/boundaries}
%\caption{The boundary curves considered in the numerical examples. }\label{Fig1}
%\end{center}
%\end{figure}
A peanut-shaped boundary with radial function
\[
r (t) =  (0.5 \cos^2 t + 0.15 \sin^2 t)^{1/2}, \quad t \in [0,2\pi], 
\]
an apple-shaped boundary with radial function
\[
r (t) =  \frac{0.45 + 0.3 \cos t -0.1 \sin 2t}{1+0.7 \cos t}, \quad t \in [0,2\pi],
\]
and a kite-shaped boundary with parametrization
\[
\b z (t) =   (\cos t + 0.7 \cos 2t,\, 1.2\sin t), \quad t \in [0,2\pi].
\]

\subsection{Example with analytic solution}

\begin{figure}[t]
\begin{center}
\includegraphics[scale=0.52]{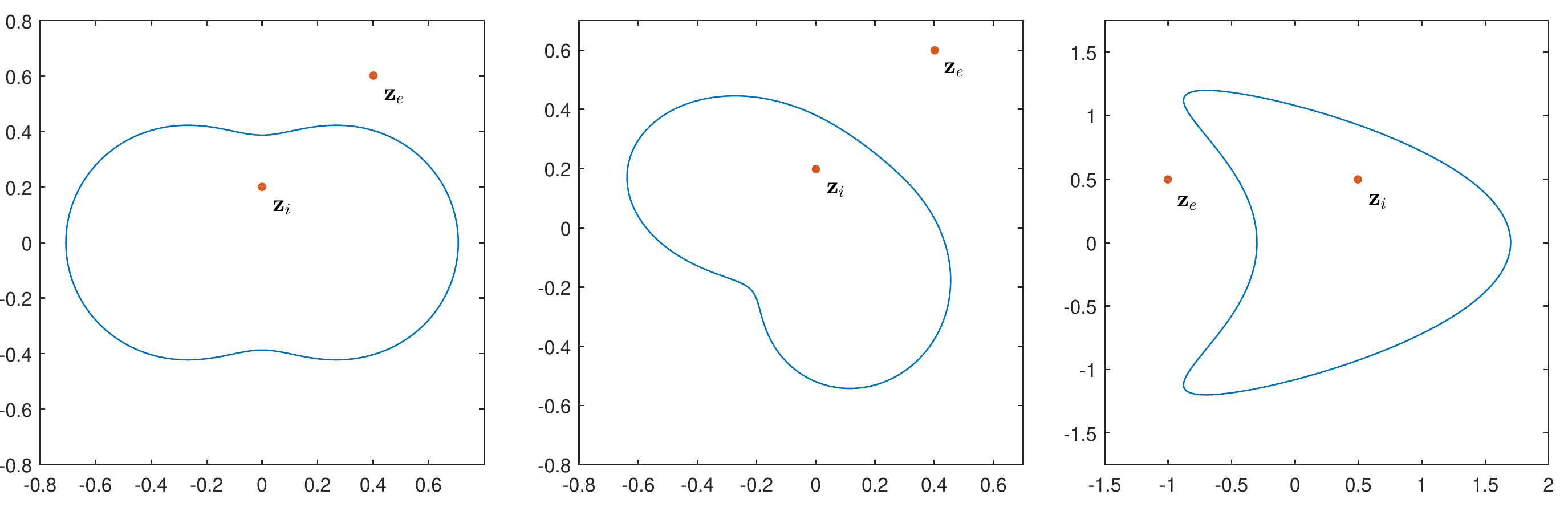}
\caption{The source points considered for the different boundary curves. }\label{Fig2s}
\end{center}
\end{figure}

We consider two arbitrary points $\b z_i \in D_i$ and $\b z_e \in D_e$ and  we define the vector-valued boundary functions
\begin{equation*}
\begin{aligned}
\b f &= [\bm \Phi_i (\b x,\b z_e ) ]_1  - [\bm \Phi_e (\b x,\b z_i )]_1 ,   & \mbox{on  } \Gamma,\\
\b g &= [\b T^i_x \bm \Phi_i (\b x,\b z_e ) ]_1  - [\b T^e_x \bm \Phi_e (\b x,\b z_i )]_1 , & \mbox{on  } \Gamma ,
\end{aligned}
\end{equation*}
where $[\cdot]_1$ denotes the first column of the tensor. Then, the fields
\[
\b u^i (\b x ) = [\bm \Phi_i (\b x,\b z_e ) ]_1   , \quad  \b x \in D_i , \qquad
\b u^e (\b x ) =  [\bm \Phi_e (\b x,\b z_i )]_1 ,  \quad  \b x \in D_e ,
\]
satisfy the Navier equations \eqref{eqNavier} and the transmission boundary conditions
\begin{equation*}
\begin{aligned}
\b u^i &=  \b u^e + \b f ,   & \mbox{on  } \Gamma,\\
\b T^i \b  u^i &= \b T^e \b u^e + \b g,  & \mbox{on  } \Gamma. 
\end{aligned}
\end{equation*}
In addition, $\b u^e$ satisfies  the Kupradze radiation condition \eqref{eqRadiation}. The exact values of the far-field patterns of $\b u^e$ considering the asymptotic behaviour of the Hankel function are given by
\[
\bm \phi^\infty_\alpha (\xhat,\b z_i ) = \beta_\alpha   \, e^{-\i k_{\alpha , e} \xhat \cdot \b z_i} [\b J_\alpha (\xhat)]_1 , \quad \alpha = p,s .
\]

To compute numerically the far-field patterns we consider the three different integral representations of the solution, meaning equations \eqref{eqSystemDire}, \eqref{eqSingleSol} and \eqref{eqDoubleSol} in order to show the efficiency of the numerical scheme. Then, the densities satisfy the corresponding systems of equations where we have to replace $\b u^{inc} |_\Gamma$ by $\b f$ and  $\b T^e \b u^{inc} |_\Gamma$ by $\b g.$

In all examples we choose the Lam\'e constants to be $\lambda_e = 1,\, \mu_e =1$ and $\rho_e =1$ in $D_e$ and $\lambda_i = 2,\, \mu_i =2$ and $\rho_i =1$ in $D_i$ and $\omega =8$ circular frequency. We consider the source points $\b z_i = (0, \, 0.2)$ and $\b z_e = (0.4, \, 0.6)$ for the peanut-shaped and the apple-shaped boundary and the points $\b z_i = (0.5, \, 0.5)$ and $\b z_e = (-1, \, 0.5)$ for the kite-shaped boundary, see \autoref{Fig2s}. 

The Tables \ref{table1}, \ref{table2} and \ref{table3} show some numerical values of the components of the  far-fields patterns at given directions. We consider different representations of the solution for the different boundary parametrizations to show that our approach is applicable in all cases. We see that the exponential convergence is clearly exhibited and we obtain the correct values related to the point source located in $D_i .$

\begin{table}[t!]
\begin{center}
\includegraphics[scale=0.8]{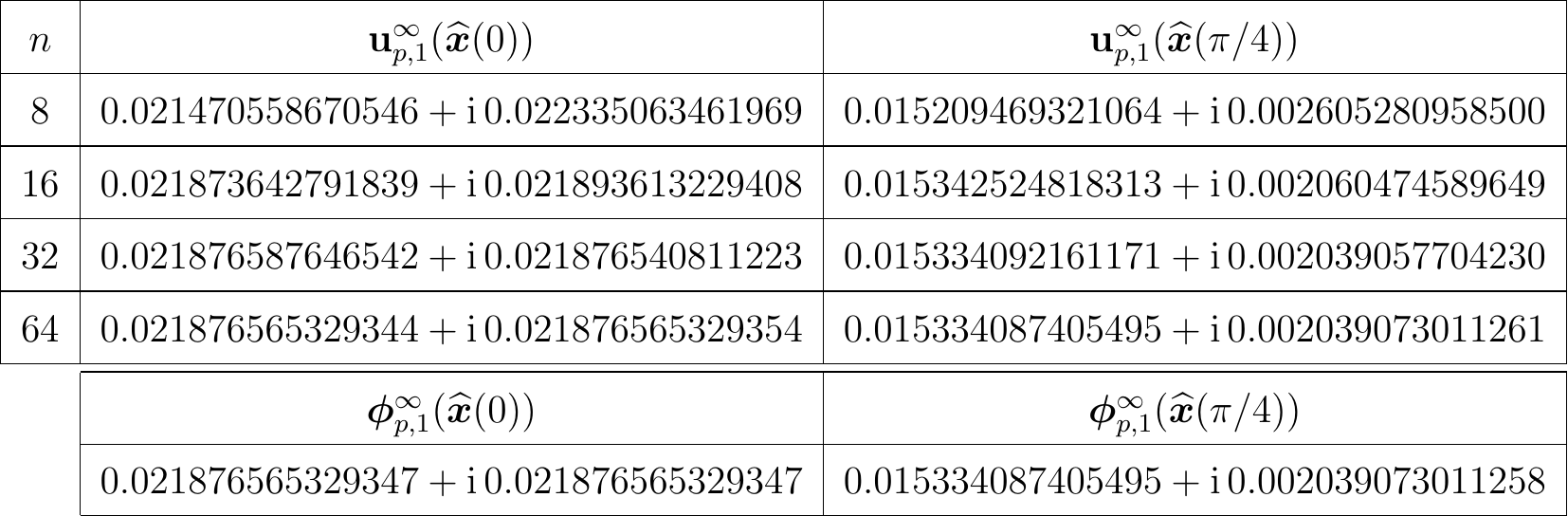}\vspace{0.3cm}
\caption{The computed and the exact longitudinal far-field for the peanut-shaped boundary considering the representation \eqref{eqFieldDire2}.}
\label{table1}
\end{center}
\end{table}

\begin{table}[t!]
\begin{center}
\includegraphics[scale=0.77]{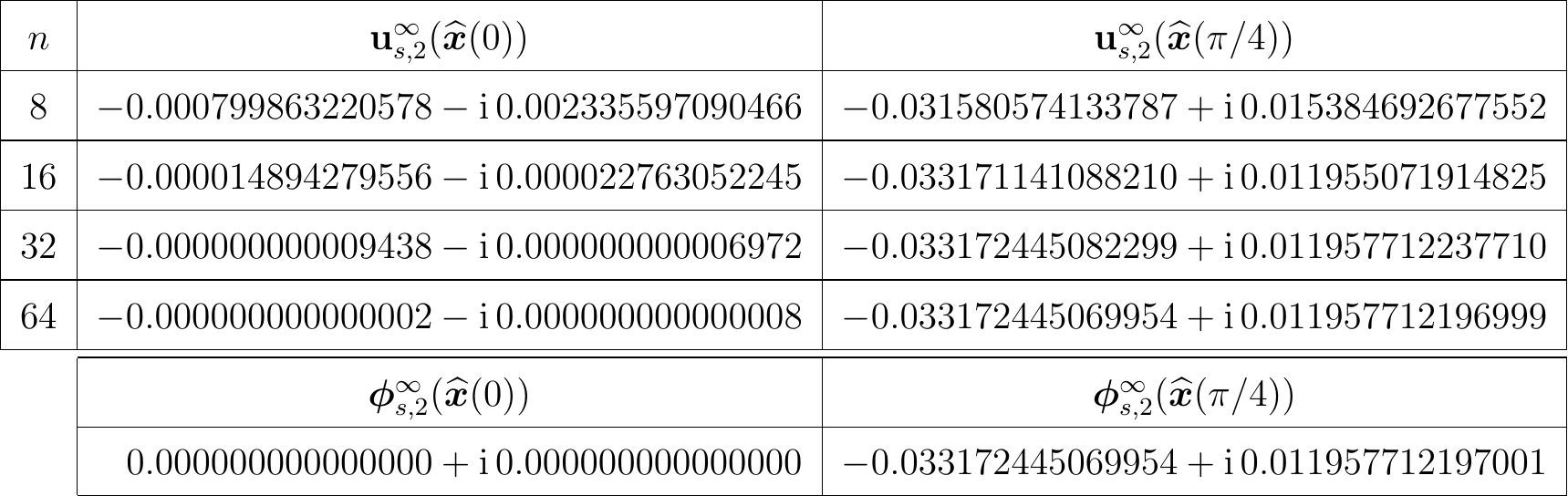}\vspace{0.3cm}
\caption{The computed and the exact transversal far-field for the apple-shaped boundary considering the representation \eqref{eqDoubleSol}.}
\label{table2}
\end{center}
\end{table}

\begin{table}[t]
\begin{center}
\includegraphics[scale=0.78]{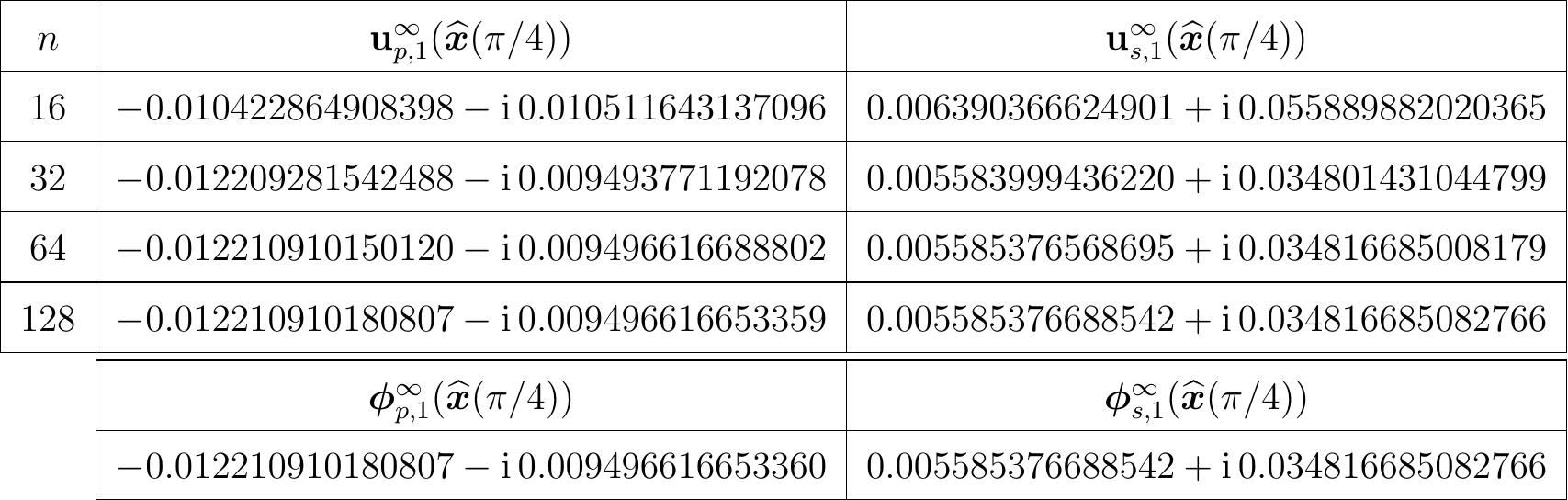}\vspace{0.3cm}
\caption{The computed and the exact far-fields for the kite-shaped boundary considering the representation \eqref{eqSingleSol}.}
\label{table3}
\end{center}
\end{table}

\subsection{The inverse problem}

To avoid an inverse crime in the following examples, the simulated far-field data were obtained by solving numerically the direct problem, replacing \eqref{eqFieldDire2} by \eqref{eqSingleSol} and considering double amount of collocation points.

We approximate the radial function $q$ by a trigonometric polynomial of the form
\[
q (t) \approx \sum_{k=0}^m a_k \cos kt + \sum_{k=1}^m b_k \sin kt , \quad t\in[0,2\pi],
\]
and we consider $2n$ equidistant points $t_j = j\pi/n, \, j=0,...,2n-1 .$

The subsystem \eqref{sub_eq} is well-posed and no special treatment is required. We solve the ill-posed linearized equation \eqref{linear_eq} by minimizing the Tikhonov functional of the corresponding discretized equation
\[
\norm{\b A \b T \b x - \b b}^2_2 + \lambda \norm{\b x}_p^p , \quad \lambda >0 .
\]
where $\b x \in \R^{(2m+1)\times 1}$ is the vector with the unknowns coefficients $a_0 ,...,a_m, b_1 ,..., b_m$ of the radial function, and $\b A \in \C^{8n \times 8n}, \, \b b \in \C^{8n \times 1}$ are given by
\begin{align*}
\b A_{kj} &= \b M_{\bb A'_3} (t_k , t_j) + \b M_{\bb B'_3} (t_k , t_j), \\
\b b_k &=  \bb C_3 (t_k) - (\b M_{\bb A_3} \cdot \ksi) (t_k)  - 
(\b M_{\bb B_3} \cdot \zita) (t_k),
\end{align*}
for $k,j = 0,...,2n-1,$ where $\b M_{\bb K}$ denotes the matrix related to the discretized kernel of the operator $\bb K.$ The multiplication matrix $\b T\in \R^{(8n)\times (2m+1)}$ stands for the trigonometric functions of the approximated radial function. Here $p\geq 0$ defines the corresponding Sobolev norm. Since $q$ has to be real valued we actually solve the following regularized equation
\begin{equation}\label{tikhonov}
\left(
\b T^\top \left(  \RE (\b A)^\top  \RE (\b A) +  \IM (\b A)^\top \IM (\b A)  \right) \b T + \lambda_k \b I_p \right) \b x = \b T^\top \left(  \RE (\b A)^\top  \RE (\b b) +  \IM (\b A)^\top \IM (\b b)  \right),
\end{equation}
on the $k$th step, where the matrix $\b I_p \in \R^{(2m+1)\times (2m+1)}$ corresponds to the Sobolev $H^p$ penalty term. We solve \eqref{tikhonov} using the conjugate gradient method.
We update the regularization parameter in each iteration step $k$ by 
\[
\lambda_k  = \lambda_0 \left( \frac23\right)^{k-1}, \quad k=1,2,...
\]
for some given initial parameter $\lambda_0 >0.$
To test the stability of the iterative method against noisy data, we add also noise to the far-field patterns with respect to the $L^2$ norm
\[
\bb{U}^\infty_\delta = \bb{U}^\infty + \delta \frac{\norm{\bb{U}^\infty}_2}{\norm{\bb{V}}_2} \bb V , 
\]
for a given noise level $\delta,$ where $ \bb V = \bb V_1 +\i \bb V_2 ,$ for $\bb V_1 , \bb V_2 \in \R^{8n \times 1}$ with components normally distributed random variables.

Already in the acoustic regime \cite{AltKre12a}, one incident wave does not provide satisfactory results, thus we have to generalize \autoref{IterationScheme} also for multiple illuminations $\b u^{inc}_l , \, l = 1,...,L .$
\begin{assumption}[Multiple illuminations]
Initially, we give an approximation of the radial function $r^{(0)}$. Then, in the $k$th iteration step:
\begin{enumerate}[i.]
\item We assume that we know $r^{(k-1)}$ and we solve the $L$ subsystems
\begin{equation}
\begin{pmatrix}
\bb{A}_1  \\
\bb{A}_2 \end{pmatrix} (r^{(k-1)}; \ksi_l ) + \begin{pmatrix}
\bb{B}_1  \\
\bb{B}_2 \end{pmatrix}(r^{(k-1)}; \zita_l ) = \begin{pmatrix}
\bb{C}_{1,l}   \\
\bb{C}_{2,l}  \end{pmatrix} (r^{(k-1)}), \quad l = 1,...,L 
\end{equation}
to obtain the densities $\ksi^{(k)}_l, \, \zita^{(k)}_l .$
\item Then, keeping the densities fixed, we solve the overdetermined version of the linearized third equation of \eqref{eqFinal}
\begin{equation*}
\begin{pmatrix}
\bb A'_3 (r^{(k-1)}; \ksi^{(k)}_1) + \bb B'_3 (r^{(k-1)}; \zita^{(k)}_1) \\
\bb A'_3 (r^{(k-1)}; \ksi^{(k)}_2) + \bb B'_3 (r^{(k-1)}; \zita^{(k)}_2) \\
\vdots \\
\bb A'_3 (r^{(k-1)}; \ksi^{(k)}_L) + \bb B'_3 (r^{(k-1)}; \zita^{(k)}_L)
\end{pmatrix}  q = \begin{pmatrix}
\bb C_{3,1} - \bb A_3 (r^{(k-1)}; \ksi^{(k)}_1) - \bb B_3 (r^{(k-1)}; \zita^{(k)}_1 ) \\
\bb C_{3,2} - \bb A_3 (r^{(k-1)}; \ksi^{(k)}_2) - \bb B_3 (r^{(k-1)}; \zita^{(k)}_2 ) \\
\vdots \\
\bb C_{3,L} - \bb A_3 (r^{(k-1)}; \ksi^{(k)}_L) - \bb B_3 (r^{(k-1)}; \zita^{(k)}_L )
\end{pmatrix}
\end{equation*}
for $q$ and we update the radial function $r^{(k)} = r^{(k-1)} +q.$
\end{enumerate}
The iteration stops when a suitable stopping criterion is satisfied.
\end{assumption}

\subsection{Numerical results}
In the following examples we choose the incident field to be a longitudinal plane wave with different incident directions given by
\[
\di_l = (\cos \tfrac{2\pi l}{L}, \, \sin \tfrac{2\pi l}{L}), \quad l=1,...,L.
\]
We choose the Lam\'e constants to be $\lambda_e = 1,\, \mu_e =1$ and $\rho_e =1$ in $D_e$ and $\lambda_i = 2,\, \mu_i =3$ and $\rho_i =1$ in $D_i$ and $\omega =8$ circular frequency. We set $n=64$ collocation points for the direct problem and $n=32$ for the inverse. The regularized equation \eqref{tikhonov} is solved for $p=1,$ meaning $H^1$ penalty term and for initial regularization parameter $\lambda_0 = 0.8.$ 

We present reconstructions for different boundary curves, different number of incident directions and initial guesses for exact and perturbed far-field data. When, we refer to noisy data, we have considered $\delta = 5\% .$  In all figures the initial guess is a circle with radius $r_0 ,$ a green solid line, the exact curve is represented by a dashed red line and the reconstructed by a solid blue line. The arrows denote the directions of the incoming incident fields.

In the first example we consider the peanut-shaped boundary. The reconstructions for $m=3$ coefficients, two incident fields and $r_0 = 0.5$ initial radius are presented in \autoref{Fig2} after 40 iterations for the exact data and 25 iterations for the noisy. In \autoref{Fig3}, we see that the reconstructions are not highly dependent on the initial guess.

In the second example, the boundary to be reconstructed is the apple-shaped. Here, we set $m=4,$ and $r_0 = 0.5.$ The reconstructions for exact data and different number of incident fields are presented in \autoref{Fig4} for 18 iterations (one incident direction) and 40 iterations (three incident directions). \autoref{Fig5} shows the effect of the initial guess for noisy data and 40 iterations.

In the last example, we choose the kite-shaped boundary. We consider $m=7$ coefficients and $r_0 = 1.5.$ In \autoref{Fig6} we see the improvement with respect to the number of incident fields for exact data, 10 iterations for three illuminations and 40 iterations for four illuminations. The dependence on the initial guess is shown in \autoref{Fig7}, for $r_0 = 1$ we needed 40 iterations and 25 for $r_0 = 1.5$, in doth cases we considered noisy data.

All examples show the feasibility of the proposed method that is also reasonably stable against noise. The results are considerably improved if we consider more that one incident wave. One could also considered more sophisticated regularization techniques and methods to compute the regularization parameter that could improve the reconstructions but are out of the scope of this paper.

%\section*{Acknowledgements} 

\begin{figure}[p]
\begin{center}
\includegraphics[scale=0.8]{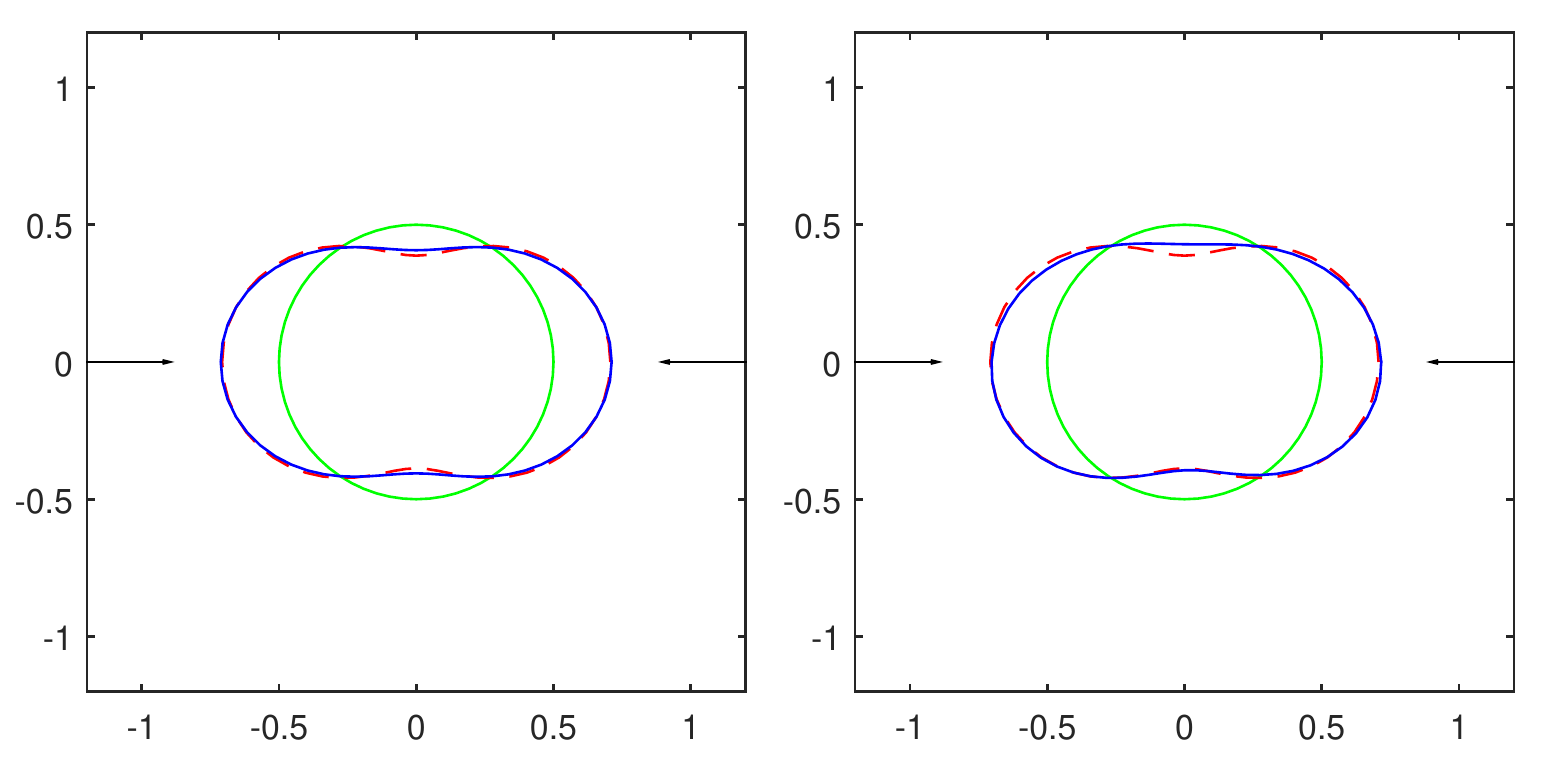}
\caption{Reconstruction of a peanut-shaped boundary for exact (left) and noisy (right) data. }\label{Fig2}
\end{center}
\end{figure}

\begin{figure}[p]
\begin{center}
\includegraphics[scale=0.8]{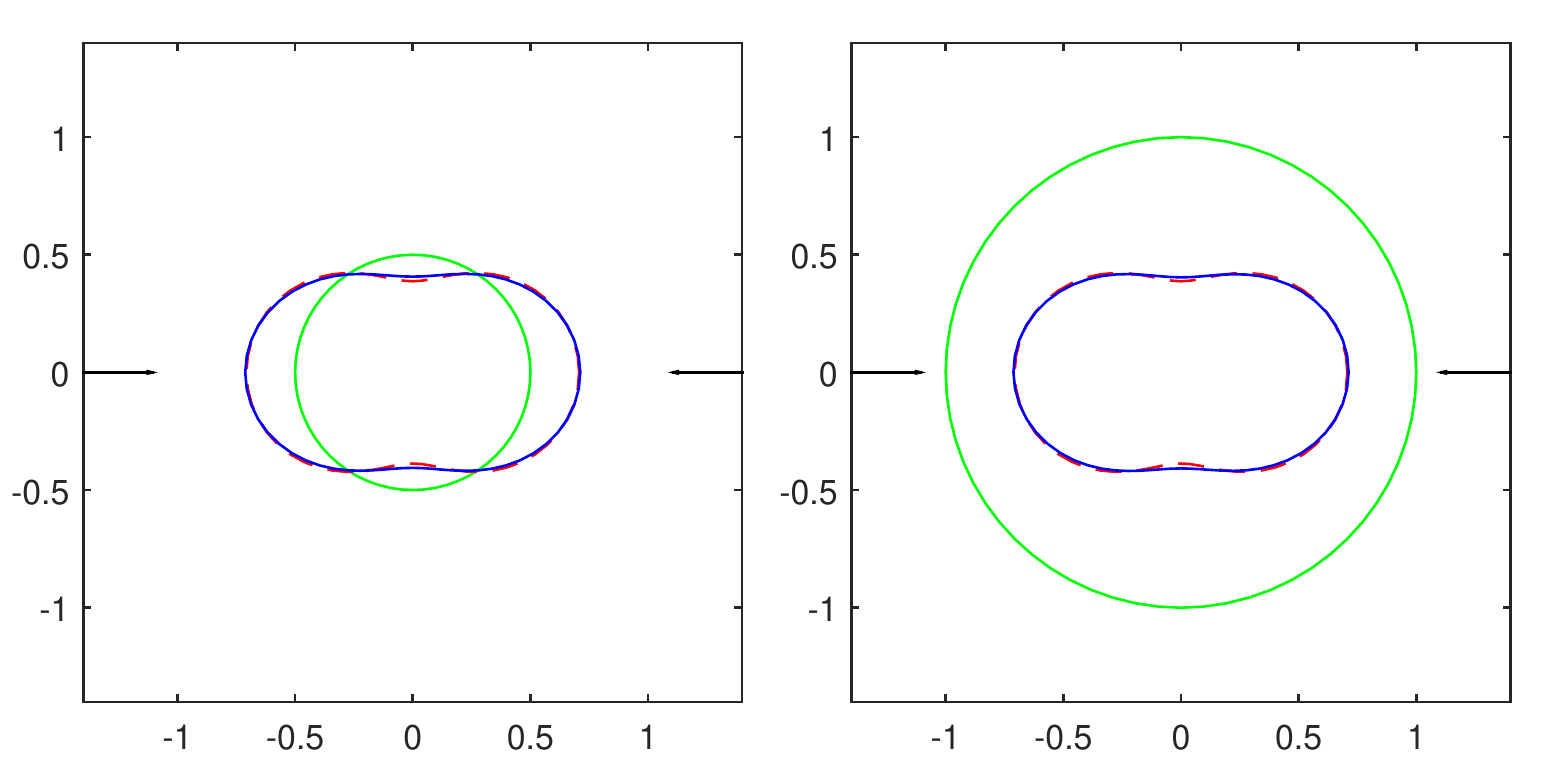}
\caption{Reconstruction of a peanut-shaped boundary for initial guess $r_0 =  0.5$ (left) and $r_0 =  1$ (right) and exact data. }\label{Fig3}
\end{center}
\end{figure}

\begin{figure}[p]
\begin{center}
\includegraphics[scale=0.58]{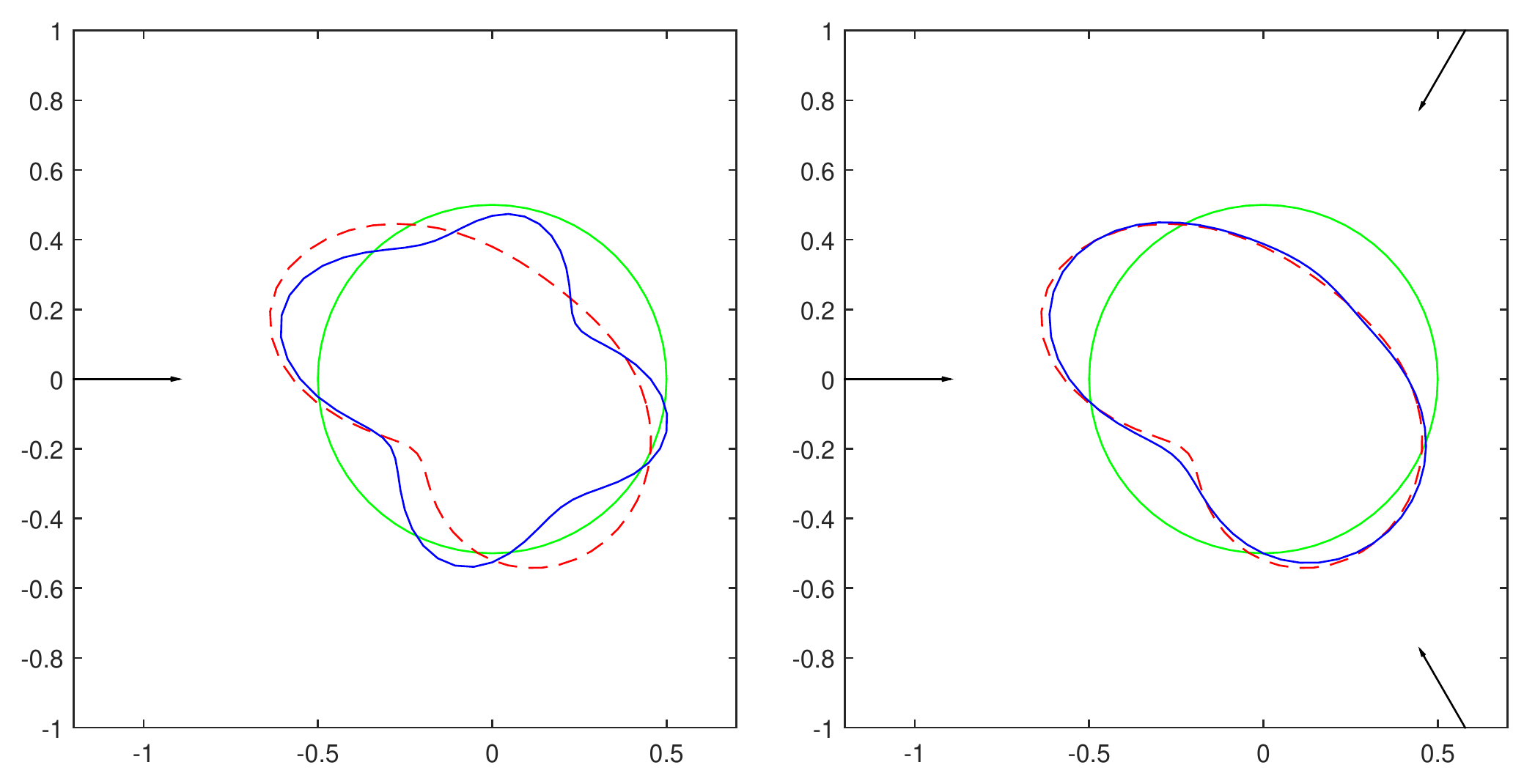}
\caption{Reconstruction of a apple-shaped boundary for one (left) and three (right) incident fields. }\label{Fig4}
\end{center}
\end{figure}

\begin{figure}[p]
\begin{center}
\includegraphics[scale=0.8]{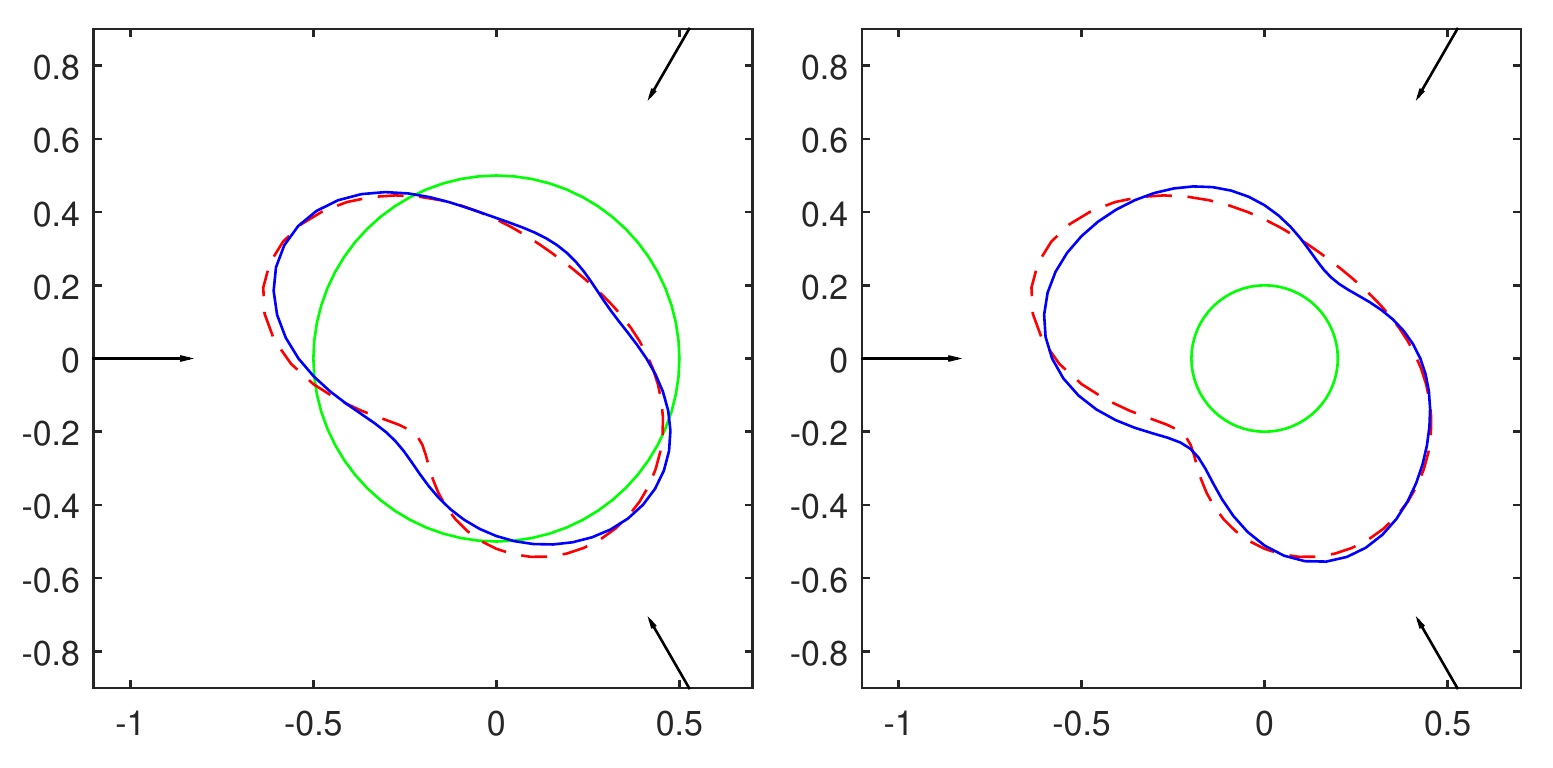}
\caption{Reconstruction of a apple-shaped boundary for initial guess $r_0 =  0.5$ (left) and $r_0 =  0.2$ (right) and noisy data. }\label{Fig5}
\end{center}
\end{figure}

\begin{figure}[p]
\begin{center}
\includegraphics[scale=0.8]{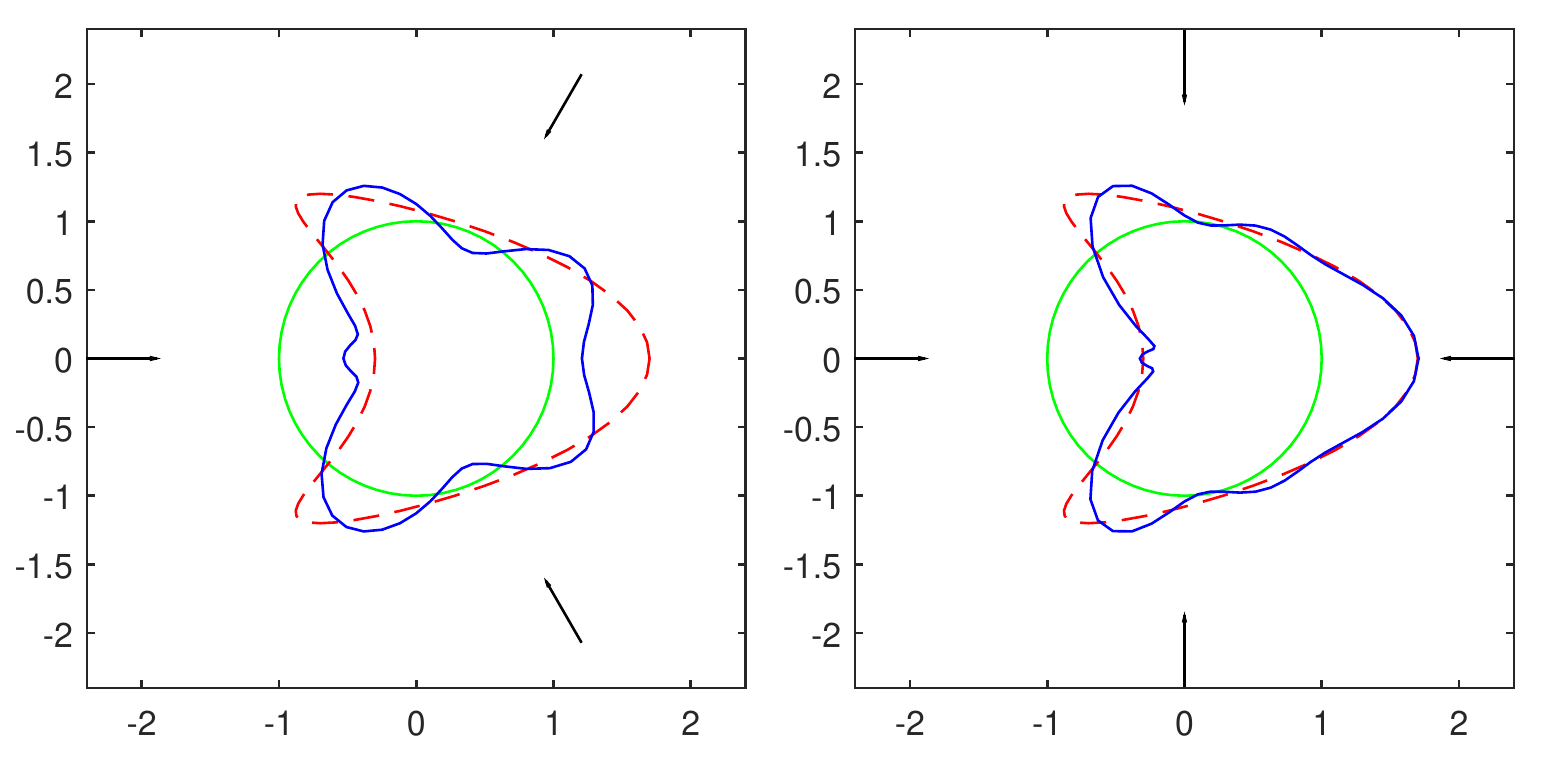}
\caption{Reconstruction of a kite-shaped boundary for three (left) and four (right) incident fields. }\label{Fig6}
\end{center}
\end{figure}

\begin{figure}[p]
\begin{center}
\includegraphics[scale=0.8]{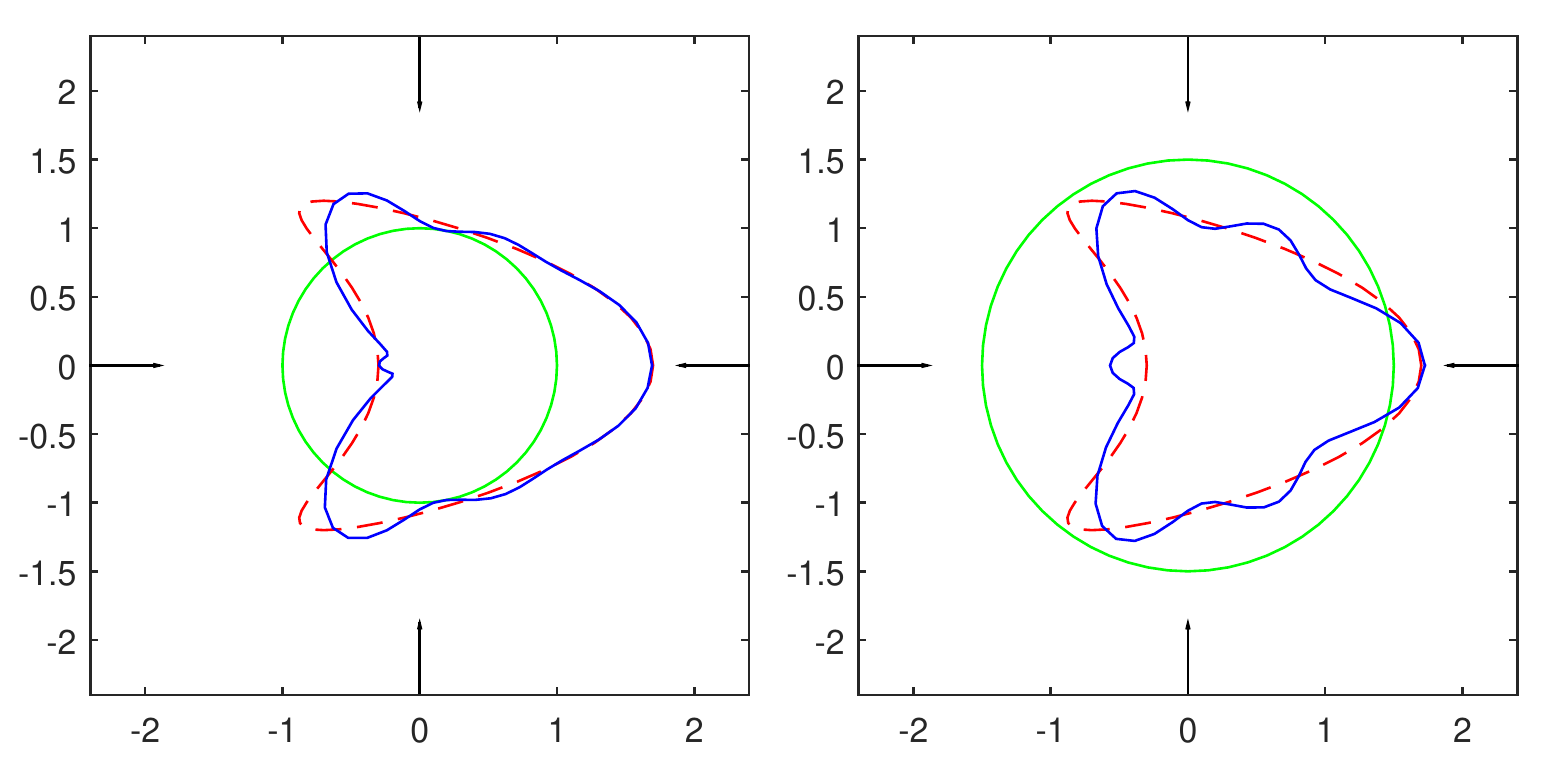}
\caption{Reconstruction of a kite-shaped boundary for initial guess $r_0 =  1$ (left) and $r_0 =  1.5$ (right) and noisy data. }\label{Fig7}
\end{center}
\end{figure}

\end{document}